\documentclass[11pt]{amsart}%
\usepackage{graphicx}
\usepackage{amscd, color}
\usepackage{amsmath}
\usepackage{amsfonts}
\usepackage{amssymb}%
\providecommand{\U}[1]{\protect\rule{.1in}{.1in}}
\providecommand{\U}[1]{\protect\rule{.1in}{.1in}}
\providecommand{\U}[1]{\protect\rule{.1in}{.1in}} \textwidth 16.3cm
\theoremstyle{plain}

\newtheorem{theorem}{Theorem}[section]

\newtheorem{corollary}[theorem]{Corollary}
\newtheorem{example}[theorem]{Example}

\newtheorem{remark}[theorem]{Remark}

\numberwithin{equation}{section}

\begin{document}
\title{Some properties of almost summing operators}
	\author[Renato Macedo]{Renato Macedo}
	\address{Departamento de Matem\'{a}tica \\
		Universidade Federal da Para\'{\i}ba \\
		58.051-900 - Jo\~{a}o Pessoa, Brazil.}
	\email{renato.burity@academico.ufpb.br}
	\author[Joedson Santos]{Joedson Santos}
	\address{Departamento de Matem\'{a}tica \\
		Universidade Federal da Para\'{\i}ba \\
		58.051-900 - Jo\~{a}o Pessoa, Brazil.}
	\email{joedson.santos@academico.ufpb.br}
	\thanks{Renato Macedo is partially supported by Capes and Joedson Santos is supported by CNPq, Grant 2019/0014 Para\'{\i}ba State Research Foundation (FAPESQ) and Public Call n. 03 Produtividade em
		Pesquisa PROPESQ/PRPG/UFPB PIA13194-2020.}
	\maketitle
	
	\begin{abstract}
		In this paper we extend the scope of three important results of the linear theory of absolutely summing operators. The first one was proved by Bu and Kranz in \cite{BK} and it asserts that a continuous linear operator between Banach spaces takes almost unconditionally summable sequences into Cohen strongly $q$-summable sequences for any $q\geq2$, whenever its adjoint is $p$-summing for some $p\geq1$. The second of them states that $p$-summing operators with hilbertian domain are Cohen strongly $q$-summing operators ($1<p,q<\infty$), this result is due to Bu \cite{Bu}. The third one is due to  Kwapie\'{n} \cite{Kwapien} and it characterizes spaces isomorphic to a Hilbert space using 2-summing operators. We will show that these results are maintained replacing the hypothesis of the operator to be $p$-summing by almost summing. We will also give an example of an almost summing operator that fails to be $p$-summing for every $1\leq p< \infty$.

	\end{abstract}
	
\section{Introduction}

The space of all continuous linear operators from a Banach space
$X$ to a Banach space $Y$ will be denoted by $\mathcal{L}(X,Y)$. The closed unit ball of $X$ is denoted by $B_{X}$ and its topological dual by $X^*$. The symbol $u^*$ represents the adjoint of a continuous linear operator $u$. Below, $r_{i}$ are the Rademacher functions and if $1\leq p\leq \infty$ we denote by $p^*$ the conjugate of $p$, i.e., $\frac{1}{p} + \frac{1}{p^*} = 1$. For a Banach space $X$, the spaces

$\ell_{p}(X):=\{  (x_{i})_{i=1}^{\infty}\subset X:\Vert (x_{i}%
)_{i=1}^{\infty}\Vert _{p}:=(\sum_{i=1}^{\infty}
\left\Vert x_{i}\right\Vert ^{p})  ^{1/p}<\infty\}$,

$\ell_{p}^{w}(X):=\{  (x_{i})_{i=1}^{\infty}\subset X:\left\Vert
(x_{i})_{i=1}^{\infty}\right\Vert _{w,p}:=\sup_{x^* \in B_{X^*}}\|(x^*(x_i))_{i=1}^\infty\|_p<\infty\}$,

$\ell_p\langle X \rangle  := \{(x_i)_{i=1}^\infty \subset X: \|(x_i)_{i=1}^\infty\|_{C,p}:= \sup_{(x_i^*)_
	{i=1}^\infty \in B_{\ell_{p^*}^w(X^*)}} \|(x_i^*(x_i))_{i=1}^\infty\|_1< \infty\}$ and

$ Rad(X):=\{(x_i)_{i=1}^\infty \subset X:\|(x_i)_{i=1}^\infty\|_{Rad(X)}:=\left({\displaystyle\int\nolimits_{0}^{1}%
}\left\Vert\sum_{i=1}^{\infty}r_i(t)x_i\right\Vert^2 dt\right)^{1/2}< \infty\}$ \\
are, respectively, the spaces of {\it absolutely, weakly, Cohen strongly (or strongly) $p$-summable and almost unconditionally summable $X$-valued sequences} (see \cite{Cohen, djt}).

If $1\leq p<\infty,$ we say that a linear operator
$u:X\rightarrow Y$ is \textit{absolutely} $p$-\textit{summing} (or $p$-\textit{summing}) if $\left(  u(x_{i})\right)
_{i=1}^{\infty}\in\ell_{p}(Y)$ whenever $(x_{i})_{i=1}^{\infty}\in\ell
_{p}^{w}(X)$. The class of absolutely $p$-summing linear operators from $X$ to $Y$ will be represented by $\Pi_{p}\left(  X,Y\right)  $ (see \cite{djt}).

In \cite{Cohen} Cohen introduced a class of operators which characterizes the $p^*$-summing adjoint operators. If $1< p\leq\infty,$ we say that a linear operator $u$ from $X$ to $Y$ is {\it Cohen strongly $p$-summing} (or {\it strongly $p$-summing})  if $\left(  u(x_{i})\right)
_{i=1}^{\infty}\in\ell_{p}\langle Y \rangle$ whenever $(x_{i})_{i=1}^{\infty}\in\ell
_{p}(X)$. The class of Cohen strongly $p$-summing linear operators from $X$ to $Y$ will be denoted by $\mathcal{D}_{p}(X,Y)$. An equivalent formulation asserts that $u \in \mathcal{L}(X,Y)$ is Cohen strongly $p$-summing if there is a constant $C\geq0$ such that
\[
\sum_{i=1}^{m}\vert y_{i}^{*}(u(x_i))\vert\leq C\cdot\left\Vert
(x_{i})_{i=1}^{m}\right\Vert _{p}\cdot\|(y^{*}_{i})_{i=1}^{m}\|_{p^*,w}
\]
for all $x_{1},...,x_{m}\in X$, $y_{1}^*,...,y_{m}^* \in Y^*$ and all positive integers $m$.

According to \cite{BBJ}, a linear operator $u\in\mathcal{L}\left(  X;Y\right)$ is called to be \textit{almost} $p$-\textit{summing}, $1\leq p<\infty,$ if there is a constant $C\geq0$ such that 
\begin{equation}\label{almostp}
\left({\displaystyle\int\nolimits_{0}^{1}%
}\left\Vert\sum_{i=1}^{m}r_i(t)u(x_i)\right\Vert^2 dt\right)^{\frac{1}{2}}  \leq C\cdot\left\|(x_i)_{i=1}^{m}\right\|_{w,p}
\end{equation}
for every $m\in \mathbb{N}$ and $x_{1},...,x_{m}\in X$. The smallest $C$ which satisfies the inequality is named the \textit{almost summing norm} of $u$ denoted by $\pi_{al.s.p}(u)$. The class of all almost summing operators from $X$ to $Y$ is denoted by $\Pi_{al.s.p}\left(  X,Y\right)$. When $p=2$, these operators are simply called \textit{almost summing} and we write $\Pi_{al.s}$ instead of $\Pi_{al.s.2}$ (see
\cite[Chapter 12]{djt}). By \cite[Proposition 12.5]{djt},
\begin{equation}
\bigcup_{1\leq p<\infty}\Pi_{p}(X,Y)\subseteq\Pi_{al.s}(X,Y).\label{inccc}%
\end{equation}

Now we will give an example of an almost summing operator that fails to be $p$-summing for every $1\leq p< \infty$. We provide this example because we have found no information in this direction in the literature to quote.

\begin{example}
	
Consider the following Banach space
\[
Y=\left(  \sum_{k=3}^{\infty}\ell_{k}\right)  _{1}:=\left\{(x_{k})_{k=3}^{\infty}; x_{k}\in \ell_{k}\ \mbox{and}\  \Vert (x_{k})_{k
	=3}^{\infty}\Vert_1={\displaystyle\sum\limits_{k=3}^{\infty}}
\left\Vert x_{k}\right\Vert_{k}\right\}%
\]
and let
\[
T_{k}:\ell_{1}\rightarrow\ell_{k}%
\]
be operators such that $T_{k}$ fails to be absolutely $k$-summing for all $k$ $=3,4,5,\ldots$ (see \cite[Proposition 5.2(v)]{Bennett}).
On the other hand, for all $k$ $=3,4,5,\ldots$ (see \cite[Proposition 5.2(iii)]{Bennett})
we have%
\[
\Pi_{k+1}\left(  \ell_{1},\ell_{k}\right)  =\mathcal{L}\left(  \ell
_{1},\ell_{k}\right)
\]
and thus%
\[
T_{k}\in\Pi_{k+1}\left(  \ell_{1},\ell_{k}\right)\subset\Pi_{al.s}\left(  \ell_{1},\ell_{k}\right).
\]

\bigskip For all $k=3,4,5,\ldots,$ let%
\[
i_{k}:\ell_{k}\rightarrow Y
\]
be the natural inclusions%
\[
i_{k}\left(  \left(  a_{j}\right)  _{j=1}^{\infty}\right)  =\left(
0,\ldots,0,\left(  a_{j}\right)  _{j=1}^{\infty},0,0,\ldots\right)  ,
\]
with $\left(  a_{j}\right)  _{j=1}^{\infty}$ at the $(k-2)  $-th
position. Note that $\left\Vert i_{k}\right\Vert =1$.

Let%
\[
T:\ell_{1}\rightarrow Y
\]
be
\[
T\left(  z\right)  =%
{\displaystyle\sum\limits_{k=3}^{\infty}}
\frac{\left(  i_{k}\circ T_{k}\right)  (z)}{C_{k}},
\]
where
\[
C_{k}=2^{k}\cdot \pi_{al.s}(T_{k}).
\]

Note that $T$ is well-defined, linear and continuous. In fact, as $\left\Vert T_{k}\right\Vert \leq \pi_{al.s}(T_{k})$ we obtain
\begin{align*}%
	{\displaystyle\sum\limits_{k=3}^{\infty}}
	\left\Vert \frac{i_{k}\circ T_{k}}{C_{k}}\right\Vert  & =%
	{\displaystyle\sum\limits_{k=3}^{\infty}}
	\frac{1}{C_{k}}\left\Vert i_{k}\circ T_{k}\right\Vert \\
	& \leq%
	{\displaystyle\sum\limits_{k=3}^{\infty}}
	\frac{1}{C_{k}}\left\Vert i_{k}\right\Vert \left\Vert T_{k}\right\Vert \\
	& =%
	{\displaystyle\sum\limits_{k=3}^{\infty}}
	\frac{\left\Vert T_{k}\right\Vert }{2^{k}\cdot \pi_{al.s}(T_{k})}\\
	& <\infty.
\end{align*}
Thus, we conclude that
\[
T:=%
{\displaystyle\sum\limits_{k=3}^{\infty}}
\frac{i_{k}\circ T_{k}}{C_{k}}%
\]
is well-defined in $\mathcal{L}\left(  \ell_{1},Y\right)  .$ 

Since $T_{k}$ fails to be absolutely $k$-summing, it is simple to prove that
\[
T\notin%
{\displaystyle\bigcup\limits_{k\in \{3,4,\ldots\}}}
\Pi_{k}\left(  \ell_{1},Y\right)=%
{\displaystyle\bigcup\limits_{p\geq1}}
\Pi_{p}\left(  \ell_{1},Y\right).
\]
In fact, it suffices to note that%
\[
\left\Vert T_{k}\left(  x\right)  \right\Vert _{k}\leq C_{k}\cdot\left\Vert T\left(  x\right)  \right\Vert _{1}.
\]

It remains to prove that $T~$is almost summing. Since $T_{k}\in\Pi_{al.s}\left(  \ell_{1},\ell_{k}\right)$, by the ideal property, we have%
\[
i_{k}\circ T_{k}\in\Pi_{al.s}\left(  \ell_{1},Y\right).
\]
Note that (since $\left\Vert i_{k}\right\Vert =1$),
\begin{align*}%
	{\displaystyle\sum\limits_{k=3}^{\infty}}
	\pi_{al.s} \left( \frac{i_{k}\circ T_{k}}{C_{k}}\right)  & =%
	{\displaystyle\sum\limits_{k=3}^{\infty}}
	\frac{1}{C_{k}}\cdot\pi_{al.s} \left(i_{k}\circ T_{k}\right)\\
	& \leq%
	{\displaystyle\sum\limits_{k=3}^{\infty}}
	\frac{1}{C_{k}}\cdot\left\Vert i_{k}\right\Vert\cdot\pi_{al.s} \left( T_{k}\right)\\
	& =%
	{\displaystyle\sum\limits_{k=3}^{\infty}}
	\frac{1}{2^{k}\cdot\pi_{al.s} \left( T_{k}\right)}\cdot\pi_{al.s} \left( T_{k}\right)\\
	& <\infty.
\end{align*}
Since $\Pi_{al.s}\left(  \ell_{1},Y\right)$ is a Banach space, we conclude that
\[%
{\displaystyle\sum\limits_{k=3}^{\infty}}
\frac{i_{k}\circ T_{k}}{C_{k}}%
\]
converges in $\Pi_{al.s}\left(  \ell_{1},Y\right)$. 
\end{example}


A Banach space $X$ has \textit{type} $p$, with $1\leq p\leq2$, if there is a constant $C\geq0$ such that, for all $m\in \mathbb{N}$ and $x_{1},...,x_{m}\in X$, 
\begin{equation*}\label{typep}
	\left({\displaystyle\int\nolimits_{0}^{1}%
	}\left\Vert\sum_{i=1}^{m}r_i(t)x_i\right\Vert^2 dt\right)^{\frac{1}{2}}  \leq C\cdot\left\|(x_i)_{i=1}^{m}\right\|_{p}.
\end{equation*}
The infimum of the $C$ is denoted by $T_p(X)$.

We should mention here Khinchin's inequality and Kahane's inequality:

\textbf{Khinchin inequality.} Let $0<p<\infty$. Then, there are positive constants $A_p$ and $B_p$ for which
	\[
	\mathrm{A}_{p}\left(  \sum_{i=1}^{m}\left\vert x_{i}\right\vert^{2}\right)
	^{\frac{1}{2}}\leq\left(  \int_{0}^{1}\left\vert \sum_{i=1}^{m}r_{i}(t)x_{i}\right\vert
	^{p}dt\right)  ^{\frac{1}{p}}\leq\mathrm{B}_{p}\left(  \sum_{i=1}^{m}\left\vert x_{i}\right\vert^{2}\right)
	^{\frac{1}{2}}
	\]
	holds, regardless of the choice of finitely many scalars $x_{1},\dots,x_{m}$.

	\textbf{Kahane inequality.} Let $0<p,q<\infty$. Then, there is a constant $\mathrm{K}%
	_{p,q}>0$ for which
	\[
	\left(  \int_{0}^{1}\left\Vert \sum_{i=1}^{m}r_{i}(t)x_{i}\right\Vert
	^{q}dt\right)  ^{\frac{1}{q}}\leq\mathrm{K}_{p,q}\left(  \int_{0}%
	^{1}\left\Vert \sum_{i=1}^{m}r_{i}(t)x_{i}\right\Vert ^{p}dt\right)
	^{\frac{1}{p}}
	\]
	holds, regardless of the choice of a Banach space $X$ and of finitely many
	vectors $x_{1},\dots,x_{m}\in X$.

Using strong tools as Pietsch domination theorem and Khinchin and Kahane inequalities, the main result proved by Bu and Kranz in \cite{BK} was:

\begin{theorem}\cite[Theorem 1]{BK} Let $X$ and $Y$ be Banach spaces and $u$ be a continuous linear operator from $X$ to $Y$. If $u^*$ is $p$-summing for some $p\geq1$, then for any $q\geq2$, $u$ takes almost unconditionally summable sequences in $X$ into members of $\ell_{q}\langle Y\rangle$.
\end{theorem}

Let $X$ be a Hilbert space. Cohen in \cite{Cohen70} has shown that
\begin{equation}\label{cohenresult}
	\Pi_{2}\left(  X,Y\right)\subseteq\mathcal{D}_{2}\left(  X,Y\right)\ \mbox{for all Banach space}\ Y.
\end{equation}
In \cite{Bu}, Bu showed that (\ref{cohenresult}) is valid with no restrictions of the parameters $p,q\in (1,\infty)$ instead of $p = q = 2$. Cohen \cite{Cohen70} also asked if (\ref{cohenresult}) characterizes spaces isomorphic to a Hilbert spaces. Kwapie\'{n} \cite{Kwapien} proved that this question has a positive answer. These important results are as follows:

\begin{theorem}\cite[Main Theorem]{Bu} \label{teorbu} Let $1<p,q<\infty$, and let $X$ be a Hilbert space and $Y$ be Banach space. Then \[\Pi_{p}\left(  X,Y\right)\subseteq\mathcal{D}_{q}\left(  X,Y\right).\]
\end{theorem}

\begin{theorem}\cite{Kwapien}\label{teorkwapien} The following properties of Banach space X are equivalent.
	
	(i) The space $X$ is isomorphic to a Hilbert space.
	
	(ii) For every Banach space $Y$, $\Pi_{2}(X,Y)\subseteq\mathcal{D}_{2}(X,Y)$.

\end{theorem}

The paper is organized as follows: In Section 2 our main result is an improvement of the result of Bu and Kranz \cite{BK} through a simpler argument than the original one. In Section 3, we extend the statement of the main result of Bu in \cite[Main Theorem]{Bu}. Finally, in Section 4, we show a Kwapie\'{n} type theorem using almost summing operators to characterize spaces isomorphic to a Hilbert space.

\section{Extending Bu-Kranz Theorem}

In the next theorem we will show that the Bu-Kranz Theorem is valid for almost $p^*$-summing operators instead of $p$-summing operators. The proof of this result is very simple and we will not use the Pietsch Domination Theorem, Khinchin's inequality and Kahane's inequality. Our main result is the following:

\begin{theorem}\label{theorinclal.s-2} Let $X$ and $Y$ be Banach spaces and $u$ be a continuous linear operator from $X$ to $Y$. If
 $u^*$ is almost $p^*$-summing for some $p\geq1$, then $u$ takes almost unconditionally summable sequences in $X$ into members of $\ell_{p}\langle Y\rangle$.
\end{theorem}

\begin{proof}
	Let $u \in \mathcal{L}(X,Y)$. For any $x_{1},...,x_{m} \in X$, any $y_{1}^*,...,y_{m}^* \in Y^*$ and taking $\theta_{i}=\mbox{sign}[y_{i}^{*}(u(x_i))]$, with $i=1,...,m$, it follows that
\begin{align*}
\sum_{i=1}^{m}\vert y_{i}^{*}(u(x_i))\vert & =  \left\vert\sum_{i=1}^{m} \theta_{i}y_{i}^{*}(u(x_i))\right\vert  \\
& =  \left\vert\sum_{i=1}^{m}u^* \left(\theta_{i}y_{i}^{*}\right)(x_i)\right\vert  \\
& = \left\vert{\displaystyle\int\nolimits_{0}^{1}%
} \sum_{i=1}^{m}r_i(t)u^{*}(\theta_{i}y_{i}^{*})\left(\sum_{i=1}^{m}r_i(t)x_i\right) dt\right\vert \\
& \leq \left( {\displaystyle\int\nolimits_{0}^{1}%
}\left\Vert\sum_{i=1}^{m}r_i(t)u^{*}(\theta_{i}y_{i}^{*})\right\Vert^{2} dt\right)^{1/2}\cdot \left( {\displaystyle\int\nolimits_{0}^{1}%
}\left\Vert\sum_{i=1}^{m}r_i(t)x_i\right\Vert^{2} dt\right)^{1/2}  \\
& \overset{(\ref{almostp})}{\leq}  \pi_{al.s.p^*}(u^{*})\cdot  \|(\theta_{i}y^{*}_{i})_{i=1}^{m}\|_{p^*,w}\cdot \|(x_{i})_{i=1}^{m}\|_{Rad(X)}\\
& =  \pi_{al.s.p^*}(u^{*})\cdot  \|(y^{*}_{i})_{i=1}^{m}\|_{p^*,w}\cdot \|(x_{i})_{i=1}^{m}\|_{Rad(X)}.
\end{align*}
Therefore, for each $(x_{i})_{i=1}^{\infty}\in Rad(X)$ and each $(y^{*}_{i})_{i=1}^{\infty}\in \ell_{p^*}^{w}(Y^*)$,
$$\sum_{i=1}^{\infty}\vert y_{i}^{*}(u(x_i))\vert<\infty.$$
That is, for each $(x_{i})_{i=1}^{\infty}\in Rad(X)$,  $(u(x_i))_{i=1}^{\infty}\in \ell_{p}\langle Y\rangle$.

\end{proof}

\begin{corollary}\label{improviment} Let $X$ and $Y$ be Banach spaces and $u$ be a continuous linear operator from $X$ to $Y$. If
	$u^*$ is almost summing, then for any $q\geq2$, $u$ takes almost unconditionally summable sequences in X into members of $\ell_{q}\langle Y\rangle$.
\end{corollary}

\begin{proof}
	Let $u \in \mathcal{L}(X,Y)$ such that $u^* \in \Pi_{al.s}(Y^*,X^*)$. Taking $p=2$ in previous theorem, we have that 
	$$\sum_{i=1}^{\infty}\vert y_{i}^{*}(u(x_i))\vert<\infty$$ for each $(x_{i})_{i=1}^{\infty}\in Rad(X)$ and each $(y^{*}_{i})_{i=1}^{\infty}\in \ell_{2}^{w}(Y^*)$.
	
	On the other hand, $\ell_{q^*}^{w}(Y^*)\subseteq\ell_{2}^{w}(Y^*)$ for $q^*\leq 2$. Thus,
	$$\sum_{i=1}^{\infty}\vert y_{i}^{*}(u(x_i))\vert<\infty$$ for each $(x_{i})_{i=1}^{\infty}\in Rad(X)$ and each $(y^{*}_{i})_{i=1}^{\infty}\in \ell_{q^*}^{w}(Y^*)$.
	
	Consequently, for each $(x_{i})_{i=1}^{\infty}\in Rad(X)$,  $(u(x_i))_{i=1}^{\infty}\in \ell_{q}\langle Y\rangle$.
	
\end{proof}
Note that from (\ref{inccc}) it is obvious that Corollary \ref{improviment} extends the statement of Bu-Kranz Theorem.

\begin{corollary} Let $X$ and $Y$ be Banach spaces such that $X$ has type $p$, with $1\leq p\leq 2$, and $u$ be a continuous linear operator from $X$ to $Y$. If $u^*$ is almost $p^*$-summing, then $u$ is Cohen strongly $p$-summing.
\end{corollary}

\begin{proof} Let $u$ be a continuous linear operator from $X$ to $Y$ such that $u^*$ is almost $p^*$-summing. Given $(x_{i})_{i=1}^{\infty}\in \ell_{p}(X)$ and since $X$ has type $p$, by \cite[Proposition 12.4]{djt}, we have that $(x_{i})_{i=1}^{\infty}\in Rad(X)$. By Theorem \ref{theorinclal.s-2}, $(u(x_i))_{i=1}^{\infty}\in \ell_{p}\langle Y\rangle$. Hence $u$ is Cohen strongly $p$-summing.
	
\end{proof}

One of the interesting problems of the theory of absolutely summing operators is to investigate when the adjoint of an operator is $p$-summing. A nice result of Cohen in \cite{Cohen} ensures that a linear operator $u:X\rightarrow Y$ is Cohen strongly $p$-summing if, and only if, $u^{*}:Y^{*}\rightarrow X^{*}$ is $p^{*}$-summing. The next result approaches this kind of problem and provides a type of reciprocal of Bu-Kranz Theorem when $X$ has type $p$, with $1\leq p\leq 2$.

\begin{theorem}\label{theoreminverse} Let $X$ and $Y$ be Banach spaces such that $X$ has type $p$, with $1\leq p\leq 2$, and $u$ be a continuous linear operator from $X$ to $Y$. If $u$ takes almost unconditionally summable sequences in X into members of $\ell_{p}\langle Y\rangle$, then $u^*$ is $p^*$-summing (or $u$ is Cohen strongly $p$-summing). 
\end{theorem}

\begin{proof} We just need to show that if $u$ takes elements of $Rad(X)$ into elements of $\ell_{p}\langle Y\rangle$, then $u^*$ takes elements of $\ell_{p^*}^{w}(Y^*)$ into elements of $\ell_{p^*}(X^*)$.
	
	Let  $(x_{i})_{i=1}^{\infty}\in \ell_{p}(X)$ and $(y^{*}_{i})_{i=1}^{\infty}\in \ell_{p^*}^{w}(Y^*)$. As $X$ has type $p$, by \cite[Proposition 12.4]{djt} we have that $(x_{i})_{i=1}^{\infty}\in Rad(X)$ and by hypothesis $(u(x_i))_{i=1}^{\infty}\in \ell_{p}\langle Y\rangle$. Thus
	$$\sum_{i=1}^{\infty}\vert y_{i}^{*}(u(x_i))\vert<\infty.$$
	So
	$$\sum_{i=1}^{\infty}\vert u^*(y_{i}^{*})(x_i)\vert<\infty.$$
	By arbitrariness of $(x_{i})_{i=1}^{\infty}\in \ell_{p}(X)$, $(u^*(y^*_i))_{i=1}^{\infty}\in \ell_{p^*}(X^*)$.

\end{proof}

\begin{corollary}\cite[Theorem 4]{BK}\label{corollarybu} Let $X$ and $Y$ be Banach spaces such that $X$ has type 2 and $u\in\mathcal{L}(X,Y)$. Then $u$ takes almost unconditionally summable sequences in X into members of $\ell_{2}\langle Y\rangle$ if, and only if, $u^*$ is 2-summing (or $u$ is Cohen strongly 2-summing).
\end{corollary}

\begin{proof} Since $X$ has type 2 it follows from \cite[Proposition 11.10]{djt} that $X^*$ has cotype 2 and by \cite[Corollary 12.7]{djt} we have that $\Pi_{al.s}\left(  Y^*,X^*\right)=\Pi_{2}\left(  Y^*,X^*\right)$. Then the result follows of the Theorem \ref{theorinclal.s-2} and Theorem \ref{theoreminverse}. 	
	
\end{proof}

Theorem \ref{theoreminverse} tell us that the hypothesis
$$(*)\ u\in \mathcal{L}(X,Y)\ \mbox{takes almost unconditionally summable sequences in}\  X\ \mbox{into members of}\ \ell_{p}\langle Y\rangle,$$
is a sufficient condition for $u^*$ to be $p^*$-summing, whenever $X$ has type $p$. On the other hand, it is also known that there are operators with a $p$-summing adjoint which are not themselves $p$-summing. So the hypothesis $(*)$ is not a sufficient condition for $u$ to be $p$-summing for some $p\geq1$. However, for some Banach spaces $X$ and $Y$ we will show that the hypothesis $(*)$ is also a sufficient condition for $u$ to be 1-summing. For this we
shall recall the notion of $\mathcal{L}_{p}$-spaces. If $\lambda>1$ and $1\leq
p\leq\infty$, a Banach space $Y$ is an $\mathcal{L}_{p}$%
\textit{-space} if every finite dimensional subspace $E$ of $Y$ is contained
in a finite dimensional subspace $F$ of $Y$ for which there is an
isomorphism $v:F\longrightarrow\ell_{p}^{\dim F}$ such that $\Vert
v\Vert\Vert v^{-1}\Vert\leq\lambda$. Also for this purpose we use the following theorem recently proved in \cite{MPS}.

\begin{theorem}\cite[Theorem 1.6]{MPS} \label{israel}
	\label{nn11}Let $X$ be a Banach space, $Y$ be an $\mathcal{L}_{p}$-space with
	$2\leq p<\infty$ and $u\in\mathcal{L}(X,Y)$. If $u^*$ is almost $p$-summing, then $u$ is 1-summing.	
\end{theorem}

\begin{theorem}\label{teorema1} Let $X$ and $Y$ be Banach spaces such that $X$ has type 2 and $Y$ is an $\mathcal{L}_{2}$-space and $u$ be a continuous linear operator from $X$ to $Y$. If $u$ takes almost unconditionally summable sequences in X into members of $\ell_{2}\langle Y\rangle$, then $u$ is 1-summing.
\end{theorem}

\begin{proof} Suppose that $u\in\mathcal{L}(X,Y)$ takes almost unconditionally summable sequences in X into members of $\ell_{2}\langle Y\rangle$. As $X$ has type 2, Corollary \ref{corollarybu} ensures that $u^*$ is 2-summing. By (\ref{inccc}), $u^*$ is almost summing. Since $Y$ is an $\mathcal{L}_{2}$-space, it follows from Theorem \ref{israel} that $u$ is 1-summing.
	
\end{proof}

\begin{remark}
	From \cite[Theorem 3.2.3(ii)]{Cohen} we know that if $Y$ is an $\mathcal{L}_{2}$-space, then $\mathcal{D}_{2}(X,Y)\subseteq\Pi_{2}(X,Y)$ for all Banach space $X$. Corollary \ref{corollarybu} and Theorem \ref{teorema1} improve this inclusion, they ensure that $\mathcal{D}_{2}(X,Y)\subseteq\Pi_{1}(X,Y)$ whenever $X$ has type 2 and $Y$ is an $\mathcal{L}_{2}$-space. 
\end{remark}

\section{Extending Bu's Theorem}

The next result is a variant of Bu's Theorem to $\mathcal{L}_{p^{*}}$-spaces, with $2\leq p<\infty$, and almost $p$-summing operators instead of Hilbert spaces and $p$-summing operators.

\begin{theorem}\label{teorbugeral} Let $2\leq p<\infty$, $1< q\leq\infty$, $X$ and $Y$ be Banach spaces such that $X$ is an $\mathcal{L}_{p^{*}}$-space. Then \[\Pi_{al.s.p}\left(  X,Y\right)\subseteq\mathcal{D}_{q}\left(  X,Y\right).\]
\end{theorem}

\begin{proof} Let $u \in \Pi_{al.s.p}(X,Y)$. Fix $x_{1},...,x_{m} \in X$ and $y_{1}^*,...,y_{m}^* \in Y^*$. Since $X$ is an $\mathcal{L}_{p^{*}}$-space, there is a finite dimensional subspace $F$ of $X$ containing $span\{x_{1},..., x_{m}\}$ and an isomorphism $v: F \rightarrow \ell^{n}_{p^{*}}$ such that $\|v\|\|v^{-1}\| \leq \lambda$.
		
	 Note that
	$$vx_{i}=\sum_{k=1}^{n} \nu_{i,k}e_{k},\ \ \ i=1,...,m,$$
where $\left(  e_{k}\right)  _{k=1}^{n}$ is the canonical basis of $\ell_{p^{*}}^{n}=\left(\ell_{p}^{n}\right)^{*}$. Then, using the monotonicity
of the $\ell_{p}$ norms, the H\"{o}lder inequality and the Khinchin inequality it follows that

	\begin{align*}
		\sum_{i=1}^{m}\vert y_{i}^{*}(u(x_i))\vert & = \sum_{i=1}^{m}\vert y_{i}^{*}(uv^{-1}v(x_i))\vert\\
		& = \sum_{i=1}^{m}\left\vert y_{i}^{*}\left(uv^{-1}\left(\sum_{k=1}^{n} \nu_{i,k}e_{k}\right)\right)\right\vert  \\
		& \leq \sum_{i=1}^{m} \left(\sum_{k=1}^{n} \vert \nu_{i,k}\vert^{p^{*}}\right)^{1/p^{*}} \cdot \left(\sum_{k=1}^{n} \vert y_{i}^{*}(uv^{-1}(e_k))\vert^p\right)^{1/p} \\
	    & \leq \sum_{i=1}^{m}  \Vert vx_{i}\Vert \cdot \left(\sum_{k=1}^{n} \vert y_{i}^{*}(uv^{-1}(e_k))\vert^2\right)^{1/2}
	    \\
	    & \leq \sum_{i=1}^{m}  \Vert v\Vert\Vert x_{i}\Vert \cdot\dfrac{1}{A_{q^*}} \left( {\displaystyle\int\nolimits_{0}^{1}%
		}\left\vert\sum_{k=1}^{n}r_k(t)y_{i}^{*}(uv^{-1}(e_k))\right\vert^{q^*} dt\right)^{1/q^*}
		\\
	    & \leq \Vert v\Vert\cdot \dfrac{1}{A_{q^*}}\cdot  \left(\sum_{i=1}^{m}  \Vert x_{i}\Vert^{q}\right)^{1/q} \cdot\left(\sum_{i=1}^{m} {\displaystyle\int\nolimits_{0}^{1}
		}\left\vert\sum_{k=1}^{n}r_k(t)y_{i}^{*}(uv^{-1}(e_k))\right\vert^{q^*} dt\right)^{1/q^*}
		\\
	& = \Vert v\Vert\cdot \dfrac{1}{A_{q^*}} \cdot \|(x_{i})_{i=1}^{m}\|_{q} \cdot\left( {\displaystyle\int\nolimits_{0}^{1}\sum_{i=1}^{m}
		}\left\vert y_{i}^{*}\left(\sum_{k=1}^{n}r_k(t)uv^{-1}(e_k)\right)\right\vert^{q^*} dt\right)^{1/q^*}.
\end{align*}
But $$\displaystyle\left\Vert
(y_{i}^{*})_{i=1}^{m}\right\Vert _{q^{*},w}:=\sup_{y^{**} \in B_{Y^{**}}}   \left(\sum_{i=1}^{m} \vert y^{**}(y_{i}^{*})\vert^{q^{*}}\right)^{1/q^{*}}=\sup_{y \in B_{Y}}   \left(\sum_{i=1}^{m} \vert y_{i}^{*}(y)\vert^{q^{*}}\right)^{1/q^{*}}.$$
Then	
\begin{align*}
	\sum_{i=1}^{m}\vert y_{i}^{*}(u(x_i))\vert & \leq \Vert v\Vert\cdot \dfrac{1}{A_{q^*}} \cdot \|(x_{i})_{i=1}^{m}\|_{q} \cdot\left( {\displaystyle\int\nolimits_{0}^{1}
		}\left(\left\Vert \sum_{k=1}^{n}r_k(t)uv^{-1}(e_k)\right\Vert\cdot  \|(y^{*}_{i})_{i=1}^{m}\|_{q^*,w}  \right)^{q^*} dt\right)^{1/q^*}\\
	& \leq \Vert v\Vert\cdot \dfrac{1}{A_{q^*}} \cdot \|(x_{i})_{i=1}^{m}\|_{q} \cdot  \|(y^{*}_{i})_{i=1}^{m}\|_{q^*,w} \cdot\left( {\displaystyle\int\nolimits_{0}^{1}
	}\left\Vert \sum_{k=1}^{n}r_k(t)uv^{-1}(e_k)\right\Vert^{q^*} dt\right)^{1/q^*}.
\end{align*}	
By Kahane's inequality we have
	\begin{align*}
	\sum_{i=1}^{m}\vert y_{i}^{*}(u(x_i))\vert 
		& \leq \Vert v\Vert\cdot \dfrac{1}{A_{q^*}} \cdot \|(x_{i})_{i=1}^{m}\|_{q} \cdot  \|(y^{*}_{i})_{i=1}^{m}\|_{q^*,w} \cdot K_{2,q^{*}}\cdot\left( {\displaystyle\int\nolimits_{0}^{1}
		}\left\Vert \sum_{k=1}^{n}r_k(t)uv^{-1}(e_k)\right\Vert^{2} dt\right)^{1/2}.
	\end{align*}
Since $u$ is almost $p$-summing, it follows that $uv^{-1}$ also is almost $p$-summing and as $\Vert
(e_{k})_{k=1}^{n}\Vert_{p,w}=1$ in $\ell_{p^{\ast}}^{n},$ we have	
	\begin{align*}
		\sum_{i=1}^{m}\vert y_{i}^{*}(u(x_i))\vert & \leq \Vert v\Vert\cdot \dfrac{1}{A_{q^*}} \cdot \|(x_{i})_{i=1}^{m}\|_{q} \cdot  \|(y^{*}_{i})_{i=1}^{m}\|_{q^*,w} \cdot K_{2,q^{*}}\cdot\pi_{al.s.p}(u)\cdot \Vert v^{-1}\Vert\cdot  \|(e_{k})_{k=1}^{n}\|_{p,w}\\
		& \leq \lambda\cdot \dfrac{1}{A_{q^*}} \cdot K_{2,q^{*}}\cdot\pi_{al.s.p}(u) \cdot \|(x_{i})_{i=1}^{m}\|_{q} \cdot  \|(y^{*}_{i})_{i=1}^{m}\|_{q^*,w} .
	\end{align*}
	Therefore, $u\in\mathcal{D}_{q}\left(  X,Y\right)$.
	
\end{proof}

As a corollary, since (\ref{inccc}) is valid and every Hilbert space is an $\mathcal{L}_{2}$-space, we have an improvement of the Bu's Theorem \cite[Main Theorem]{Bu}:

\begin{corollary}[Bu's Theorem improved] Consider $1<q\leq \infty$. Let $X$ be an $\mathcal{L}_{2}$-space, $Y$ be a Banach space. Then
	\begin{equation*}
		\Pi_{al.s}(X,Y)\subseteq\mathcal{D}_{q}(X,Y).%
	\end{equation*}
\end{corollary}

In \cite[Corollary 9.12(b)]{djt} there is a characterization of closed subspace of $L_{r}$, with $1\leq r<\infty$, which is: a Banach space $Y$ is an isomorphic to a closed subspace of $L_{r}$ if, and only if, for all Banach space $X$
\[\mathcal{D}_{r^{*}}\left(  X,Y\right)\subseteq\Pi_{r}\left(  X,Y\right).\] 

In the next corollary, we show a coincidence result among the classes 
of the Cohen strongly $r^{*}$-summing, absolutely $r$-summing and almost summing operators.

\begin{corollary} Let $X$ be an $\mathcal{L}_{2}$-space and $Y$ be an isomorphic to a closed subspace of $L_{r}$, with $1\leq r<\infty$. Then 
	\[\Pi_{al.s}\left(  X,Y\right)=\mathcal{D}_{r^{*}}\left(  X,Y\right)=\Pi_{r}\left(  X,Y\right).\]
	
\end{corollary}	

\begin{proof} Let $X$ be an $\mathcal{L}_{2}$-space. So considering $p=2$ and $q=r^{*}$ in Theorem \ref{teorbugeral} we have that 
	\[\Pi_{al.s}\left(  X,Y\right)\subseteq\mathcal{D}_{r^{*}}\left(  X,Y\right).\]
	Since $Y$ is an isomorphic to a closed subspace of $L_{r}$, it follows from \cite[Corollary 9.12(b)]{djt} and (\ref{inccc}) that $\mathcal{D}_{r^{*}}\left(  X,Y\right)\subseteq\Pi_{r}\left(  X,Y\right)\subseteq\Pi_{al.s}\left(  X,Y\right).$
	Therefore,
	\[\Pi_{al.s}\left(  X,Y\right)=\mathcal{D}_{r^{*}}\left(  X,Y\right)=\Pi_{r}\left(  X,Y\right).\]

\end{proof}


\section{Extending Kwapie\'{n}'s Theorem}

The following theorem is the main result of this section. It is an improvement of Kwapie\'{n}’s theorem in \cite{Kwapien}.

\begin{theorem}\label{teorkwapienmelhorado} The following properties of Banach space X are equivalent.
	
	(i) The space $X$ is isomorphic to a Hilbert space.
	
	(ii) For every $1<q\leq \infty$ and every Banach space $Y$, $\Pi_{al.s}(X,Y)\subseteq\mathcal{D}_{q}(X,Y)$.
	
	(iii) For every Banach space $Y$, 
	$\Pi_{al.s}(X,Y)\subseteq\mathcal{D}_{2}(X,Y)$.
\end{theorem}

\begin{proof}$(i)\Rightarrow(ii)$ If $X$ is isomorphic to a Hilbert space, then it is an $\mathcal{L}_{2}$-space. By Theorem \ref{teorbugeral} it follows that $\Pi_{al.s}(X,Y)\subseteq\mathcal{D}_{q}(X,Y)$ for all Banach space $Y$ and all $1<q\leq \infty$.
	
	$(ii)\Rightarrow(iii)$ It is obvious.
	
	$(iii)\Rightarrow(i)$ For every Banach space $Y$, suppose that $u\in \Pi_{2}(X,Y)$. By (\ref{inccc}), $u\in \Pi_{al.s}(X,Y)$. Now using the hypothesis, it follows that $u\in \mathcal{D}_{2}(X,Y)$ and hence Kwapie\'{n}’s theorem implies that $X$ is isomorphic to a Hilbert space. 
	
\end{proof}

\begin{corollary} If $X$ is isomorphic to a Hilbert space and $Y$ has type 2, then
	 \[\Pi_{al.s}\left(  X,Y\right)=\mathcal{D}_{2}\left(  X,Y\right).\]
	
\end{corollary}	

\begin{proof} Since $X$ is isomorphic to a Hilbert, Theorem \ref{teorkwapienmelhorado} ensures that $\Pi_{al.s}\left(  X,Y\right)\subseteq\mathcal{D}_{2}\left(  X,Y\right)$. Now as $Y$ has type 2, it follows from \cite[Corollary 12.21]{djt} that $\mathcal{D}_{2}\left(  X,Y\right)\subseteq\Pi_{al.s}\left(  X,Y\right).$ Thus, the proof is complete.

\end{proof}

\end{document}